\documentclass[11pt]{amsart}

\usepackage{tikz}

\usepackage{amsmath,amssymb,amsfonts,amsthm,enumerate}

\newtheorem{theorem}{Theorem}[section]
\newtheorem{proposition}[theorem]{Proposition}
\newtheorem{lemma}[theorem]{Lemma}
\newtheorem{corollary}[theorem]{Corollary}

\theoremstyle{definition}
\newtheorem{definition}[theorem]{Definition}

\newtheorem{remark}[theorem]{Remark}

\newcommand{\ZZ}{ \ensuremath{\mathbb{Z}}}

\newcommand{\lk}{{\mathrm{lk}}}

\def\cocoa{{\hbox{\rm C\kern-.13em o\kern-.07em C\kern-.13em o\kern-.15em A}}}

\newcommand{\FF}{ \ensuremath{\mathbb{F}}}

\begin{document}

\title[On stacked triangulated manifolds]{On stacked triangulated manifolds}

\author{Basudeb Datta}
\address{Basudeb Datta, Department of Mathematics, Indian Institute of Science, Bangalore 560\,012, India. Email: {\em dattab@math.iisc.ernet.in}}

\author{Satoshi Murai}
\address{Satoshi Murai, Department of Pure and Applied Mathematics,
Graduate School of Information Science and Technology, Osaka University, Toyonaka, Osaka, 560-0043, Japan. Email: {\em s-murai@ist.osaka-u.ac.jp }}


\vspace{-3mm}

\begin{abstract}
We prove two results on stacked triangulated manifolds in this paper: (a) every stacked triangulation of a
connected manifold with or without boundary is obtained from a simplex or the boundary of a simplex by certain
combinatorial operations; (b) in dimension $d \geq 4$,
if $\Delta$ is a tight connected closed homology $d$-manifold whose $i$th homology vanishes for $1 < i < d-1$, then $\Delta$ is a stacked triangulation of a manifold.
These results give affirmative
answers to questions posed by Novik and Swartz and by Effenberger.
\end{abstract}

\maketitle

\vspace{-5mm}

\noindent {\small {\em MSC 2010\,:} 57Q15, 57R20, 05C40.

\noindent {\em Keywords:} Stacked manifolds; Triangulations of 3-manifolds; Tight triangulations.}

\vspace{-3mm}

\section{Introduction}

Stacked triangulations of spheres are of fundamental interest, in particular in the study of convex polytopes and
triangulations of spheres. Recently, the notion of stackedness was extended to triangulations of manifolds in
\cite{MN}. In this paper, we prove two results on stacked triangulations of manifolds.

We say that a  simplicial complex $\Delta$ is a {\em triangulation} of a manifold $M$ if its geometric carrier
$|\Delta|$ is homeomorphic to $M$. A triangulation of a $d$-manifold with non-empty boundary is said to be {\em
stacked} if all its interior faces have dimension $\geq d-1$. A triangulation of a closed $d$-manifold (that is,
a compact $d$-manifold without boundary) is said to be {\em stacked} if it is the boundary of a stacked triangulation of
a $(d+1)$-manifold. A triangulation of a $d$-manifold is said to be {\em locally stacked} if each vertex link is
a stacked triangulation of the $(d-1)$-sphere or the $(d-1)$-ball.

Kalai \cite{Ka} proved that, for $d \geq 4$, every locally stacked triangulation of a connected closed
$d$-manifold can be obtained from the boundary of a $(d+1)$-simplex by certain combinatorial operations. This
result does not hold for $3$-manifolds since there are triangulations of $3$-manifolds which are locally stacked
but cannot be obtained by these operations (see e.g.\ \cite[Example 6.2]{BDS2}). On the other hand, since the
stackedness and the locally stackedness are equivalent in dimension $\geq 4$ \cite{BDkst,MN}, Kalai's result also
characterizes stacked triangulations of connected closed manifolds of dimension $\geq 4$. We give a similar
characterization for stacked triangulations of manifolds with boundary (Theorem
\ref{theo:main.3.12}). As a consequence, we generalize the result of Kalai to stacked triangulations of closed
manifolds of dimension $\geq 2$ (Corollary \ref{cor:main.3.13}). This result and a recent result of Bagchi
\cite{Bag} solve a question posed  by Novik and Swartz \cite[Problem 5.3]{NS}.

Our second result is about an equivalence of tightness and tight-neighborliness. Let $\widetilde H_i(\Delta;
\FF)$ be the $i$th reduced homology group of a topological space (or a simplicial complex) $\Delta$ with
coefficients in a field $\FF$. The number $\beta_i(\Delta; \FF):= \dim_\FF \widetilde H_i(\Delta; \FF)$ is called
the {\em $i$th Betti number} of $\Delta$ with respect to $\FF$. For a simplicial complex $\Delta$ on the vertex
set $V$, we write $\Delta[W] = \{\alpha \in \Delta: \alpha \subseteq W\}$ for its induced subcomplex on $W
\subseteq V$. A simplicial complex $\Delta$ on the vertex set $V$ is said to be {\em $\FF$-tight} if it is
connected and the natural map $\widetilde H_i(\Delta[W];\FF) \to \widetilde H_i(\Delta;\FF)$ induced by the
inclusion map is injective for all $W \subseteq V$ and for all $i\geq 0$. We simply say that a simplicial complex
is {\em tight} if it is $\FF$-tight for some field $\FF$. See \cite{Ku,KL} for background and motivations of
tightness.  A simplicial complex is said to be {\em neighborly} if each pair of vertices forms a face.
Note that a tight simplicial complex is  neighborly (cf.\ \cite{BDtight}).

An $n$-vertex triangulation $\Delta$ of a closed manifold of dimension $d \geq 3$ is said to be {\em
tight-neighborly} if $\binom{n -d -1}{2} = \binom{d+ 2}{2} \beta_1(\Delta; \ZZ/2\ZZ)$. This condition is known to
be equivalent to saying that $\Delta$ is stacked and neighborly (cf.\ Section 5). Tight-neighborliness was introduced by Lutz, Sulanke and Swartz. They conjectured that
tight-neighborly triangulations are $(\ZZ/2 \ZZ)$-tight \cite[Conjecture 13]{LSS}. The conjecture was solved by
Effenberger \cite[Corollary 4.4]{Ef} in dimension $\geq 4$ and by Burton, Datta, Singh and Spreer \cite[Corollary
1.3]{BDSS} in dimension $3$. On the other hand, Effenberger \cite[Question 4.5]{Ef} asked if the converse of this
property holds for triangulations of connected sums of $S^{d-1}$-bundles over $S^1$ when $d \geq 4$.

We answer Effenberger's question affirmatively.
More generally, we prove that, in dimension $d \geq 4$, every tight, closed, orientable, $\FF$-homology $d$-manifold with $\beta_i(\Delta;\FF)=0$ for $1 < i <d-1$, is tight-neighbourly (Corollary \ref{cor4.4}). This result and Effenberger's result say
that, for triangulations of connected sums of $S^{d-1}$-bundles over $S^1$ with $d \geq 4$, tightness is
equivalent to tight-neighborliness. Also, since tight-neighborly triangulations are vertex minimal
triangulations, the result solves a special case of a conjecture of K\"uhnel and Lutz \cite[Conjecture 1.3]{KL}
which states that every tight combinatorial triangulation is vertex minimal.

This paper is organized as follows. In Section 2, we give a few basic definitions.
In Section 3, we define an analogue of a combinatorial handle addition for homology manifolds with boundary and study its basic properties. In Section 4, we present a
combinatorial characterization of stacked triangulated manifolds with and without boundary. In Section 5, we study the stackedness of tight triangulations.

\section{Preliminaries}

Recall that a {\em simplicial complex} is a collection of finite sets (sets of {\em vertices}) such
that every subset of an element is also an element. All simplicial complexes here are finite and non-empty.  For
$i \geq 0$, the elements of size $i+1$ are called the {\em $i$-faces} (or {\em faces of
dimension $i$}) of the complex. The empty set $\emptyset$ is a face (of dimension $-1$) of every simplicial
complex. For a simplicial complex $\Delta$, let $f_i(\Delta)$ be the number of $i$-faces of $\Delta$. The maximum
$k$ such that $\Delta$ has a $k$-simplex is called the {\em dimension} of $\Delta$ and is denoted by
$\dim(\Delta)$. A maximal face (under inclusion) in $\Delta$ is called a {\em facet} of $\Delta$. If $\sigma$ is
a face of $\Delta$ then the {\em link} of $\sigma$ in $\Delta$ is the subcomplex
\begin{align}
\lk_\Delta(\sigma) & =\{ \tau \setminus \sigma : \sigma \subseteq \tau \in \Delta\}. \nonumber
\end{align}

We are mainly interested in triangulations of manifolds,
but we actually consider slightly more general objects called homology manifolds.
For a field $\FF$, a simplicial complex $S$ of dimension $d$ is said to be an {\em $\FF$-homology $d$-sphere} if,
for each face $\sigma$ of dimension $i \geq -1$, $\lk_S(\sigma)$ has the same $\FF$-homologies as the
$(d-i-1)$-sphere. A simplicial complex $B$ of dimension $d$ is said to be an {\em $\FF$-homology $d$-ball} if (i)
$B$ has trivial reduced $\FF$-homologies, (ii) for each face $\sigma$ of dimension $i \leq d-1$, the reduced
$\FF$-homologies of $\lk_B(\sigma)$ are trivial or the same as those of the ($d-i-1)$-sphere and (iii) the
boundary
\begin{align}
\label{boundary}
\partial B =\{\sigma \in B \, : -1<\dim(\sigma) <d \mbox{ and } \widetilde H_{d-\dim(\sigma)-
1}(\lk_B(\sigma); \FF)=0 \}\cup \{\emptyset\}
\end{align}
is an $\FF$-homology $(d-1)$-sphere. A simplicial complex is said
to be an {\em $\FF$-homology $d$-manifold} if each vertex link is either an $\FF$-homology $(d-1)$-sphere or an $\FF$-homology
$(d-1)$-ball. Note that a triangulation of a $d$-manifold is an $\FF$-homology $d$-manifold for every field $\FF$.

In the rest of the paper,
we fix a field $\FF$ and by a {\em homology manifold/ball/sphere} we shall mean an $\FF$-homology manifold/ball/ sphere.
We define the {\em boundary} $\partial \Delta$ of a homology $d$-manifold $\Delta$ in the same way as in \eqref{boundary}.
If $\partial \Delta=\{\emptyset\}$, then $\Delta$ is called a {\em closed homology $d$-manifold} (or a {\em
homology $d$-manifold without boundary}), otherwise $\Delta$ is called a {\em homology $d$-manifold with
boundary}. If $\Delta$ is a homology $d$-manifold with boundary, then $\partial \Delta$ becomes a closed homology
$(d-1)$-manifold.
We say that a connected, closed, $\FF$-homology $d$-manifold $\Delta$ is {\em $\FF$-orientable} (or simply, {\em orientable}) if $\widetilde H_d(\Delta;\FF)$ $\cong\FF$.
Similarly, a connected homology $d$-manifold $\Delta$ with boundary is \textit{orientable} if $H_d(\Delta,\partial \Delta;\FF) \cong \FF$.
We note the following easy fact.

\begin{lemma} \label{lemma:2.1}
Let $\Delta$ be an orientable homology $d$-manifold with boundary. If $\Delta$ has trivial reduced homologies then $\Delta$ is a homology $d$-ball.
\end{lemma}

\begin{proof}
It is clear that $\Delta$ satisfies conditions (i) and (ii) of homology balls. The fact that $\partial \Delta$ is
a homology $(d-1)$-sphere follows from the long exact sequence of the pair $(\Delta,\partial \Delta)$ and the
Poincar\'{e}--Lefschetz duality \cite[Theorem 6.2.19]{Sp}.
\end{proof}

We define the stackedness and the locally stackedness for homology manifolds in the same way as for
triangulations of manifolds. Clearly, a stacked homology manifold is locally stacked. Since any stacked homology
ball (resp., sphere) is a combinatorial ball (resp., sphere), it follows that every (locally) stacked homology manifold is a combinatorial manifold, i.e., a PL triangulation of an actual manifold.
Thus we simply call them (\textit{locally$)$ stacked manifolds} (\textit{with or without boundary}).

Next, we recall Walkup's class $\mathcal H^d$. Let $\Delta$ be a connected, closed, homology manifold and let
$\sigma$ and $\tau$ be facets of $\Delta$. We say that a bijection $\psi: \sigma \to \tau$ is {\em admissible} if
$\lk_\Delta(v) \cap \lk_\Delta(\psi(v))=\{\emptyset\}$ for each vertex $v \in \sigma$. Note that, for the
existence of an admissible bijection $\psi:\sigma \to \tau$, $\sigma$ and $\tau$ must be disjoint. For an
admissible bijection $\psi : \sigma \to \tau$, let $\Delta^\psi$ be the simplicial complex obtained from
$\Delta\setminus\{ \sigma, \tau\}$ by identifying $v$ and $\psi(v)$ for all $v \in \sigma$. The simplicial
complex $\Delta^\psi$ is said to be obtained from $\Delta$ by a {\em combinatorial handle addition}.

\begin{definition}[Walkup's class $\mathcal H^d$] \label{def:Walkup}
Let $d \geq 3$ be an integer. We recursively define the class $\mathcal H^d(k)$ as follows.
\begin{enumerate}[{\rm (a)}]
\item $\mathcal H^d(0)$ is the set of stacked triangulations of the $(d-1)$-sphere.

\item A simplicial complex $\Delta$ is in $\mathcal H^d(k +1)$ if it is obtained from a member of $\mathcal
H^d(k)$ by a combinatorial handle addition.
\end{enumerate}
The Walkup's class $\mathcal H^d$ is the union $\mathcal H^d=\bigcup_{k \geq 0} \mathcal H^d(k)$.
\end{definition}

Kalai \cite[Corollary 8.4]{Ka} proved that every connected, locally stacked, closed,  $d$-manifold
is a member of Walkup's class $\mathcal{H}^{d+1}$ when $d \geq 4$, and as a consequence it follows that
$\mathcal{H}^{d+1}$ is the
set of all (locally) stacked, connected, closed, $d$-manifolds for $d\geq 4$. Although Kalai's
result is not true for $d=3$ (see e.g.\ \cite[Example 6.2]{BDS2}), we prove that $\mathcal{H}^{d+1}$ is the set
of all connected, stacked, closed, $d$-manifolds  for $d\geq 2$.

\section{Simplicial handle addition}

In this section, we define an analogue of combinatorial handle additions for homology manifolds with boundary.
Some statements in this section will be easy for experts,
but we include all the proofs for the sake of completeness.

All homologies are with coefficients in an arbitrary field $\FF$, which is fixed throughout,
and $\widetilde{H}_i(\Delta;\FF)$ and $\beta_i(\Delta;\FF)$ will be denoted by $\widetilde{H}_i(\Delta)$ and $\beta_i(\Delta)$, respectively.
Let $\Delta$
be a homology $d$-manifold with boundary on the vertex set $V$ and let $\sigma$ and $\tau$ be facets of $\partial
\Delta$. We say that a bijection $\psi: \sigma \to \tau$ is {\em admissible} if, for every vertex $v \in \sigma$,
$\lk_\Delta(v) \cap \lk_\Delta(\psi(v))=\{\emptyset\}$.
For an admissible bijection $\psi: \sigma \to \tau$, let $\Delta^{\overline \psi}$ be the simplicial complex obtained from $\Delta$ by identifying $v$ and $\psi(v)$
for all $v \in \sigma$.
(The main difference between $\Delta^\psi$ and $\Delta^{\overline \psi}$ is that we do not remove $\sigma$ and $\tau$ for the definition of $\Delta^{\overline \psi}$.)
Thus, if we define  $\psi_+$ to be the map  from $V$ to $V \setminus \sigma$ by $\psi_+(v)=\psi(v)$ if $v \in \sigma$ and $\psi_+(v)=v$ otherwise, then we can consider $\Delta^{\overline \psi}$ on the vertex set $V \setminus \sigma$ as
\begin{align}
\Delta^{\overline \psi} & =\{ \psi_+(\alpha): \alpha \in \Delta\}. \nonumber
\end{align}
If $\Delta$ is connected, then we say that $\Delta^{\overline \psi}$ is obtained from
$\Delta$ by a {\em simplicial handle addition}. If $\Delta$ has two connected components $\Delta_1$ and
$\Delta_2$ and if $\sigma \in \Delta_1$ and $\tau \in \Delta_2$, then we write $\Delta^{\overline \psi}=\Delta_1
\cup_{\psi} \Delta_2$ and call it a {\em simplicial connected union} of $\Delta_1$ and $\Delta_2$. Below we give
some basic properties of $\Delta^{\overline \psi}$.
For $\sigma \in \Delta$, we write $\overline{\sigma}$ for the simplicial complex having a single facet $\sigma$.

\begin{lemma} \label{lemma:3.1}
Let $\Delta$ and $\Gamma$ be two homology $d$-balls. If $\Delta \cap \Gamma = \partial\Delta\cap \partial\Gamma =
\overline{\alpha}$, where $\alpha$ is a $(d-1)$-simplex,  then $\Delta \cup \Gamma$ is a homology $d$-ball.
\end{lemma}

\begin{proof}
We use induction on $d$. The statement is obvious when $d=1$. Suppose $d >1$. Since $\Delta \cap \Gamma=
\overline \alpha$, the exactness of the Mayer--Vietoris sequence implies that $\Delta \cup \Gamma$ has trivial
reduced homology. Let $v$ be a vertex of $\Delta \cup \Gamma$. If $v \not\in \alpha$ then $\lk_{\Delta \cup
\Gamma}(v)$ is equal to either $\lk_{\Delta}(v)$ or $\lk_{\Gamma}(v)$ and hence a homology $(d-1)$-sphere or
$(d-1)$-ball. If $v\in \alpha$ then $v\in \partial \Delta \cap \partial \Gamma$ and hence $\lk_{\Delta}(v)$ and
$\lk_{\Gamma}(v)$ are homology $(d-1)$-balls and $\lk_{\Delta}(v) \cap \lk_{\Gamma}(v)= \overline{\alpha\setminus
\{v\}}$. Since $\lk_{\Delta \cup \Gamma}(v)= \lk_{\Delta}(v) \cup \lk_{\Gamma}(v)$, $\lk_{\Delta \cup \Gamma}(v)$
is a homology $(d-1)$-ball by the induction hypothesis. The lemma now follows from Lemma $\ref{lemma:2.1}$.
\end{proof}

It follows from Lemma \ref{lemma:3.1} that the simplicial connected union of two homology $d$-balls is a homology
$d$-ball.

\begin{lemma} \label{lemma:3.2}
For $d \geq 2$, let $\Delta$ be a $($not necessary connected$)$ homology $d$-manifold with boundary.
Let $\sigma$ and $\tau$ be two facets of $\partial \Delta$. If $\psi:\sigma \to \tau$ is  an admissible bijection then
\begin{enumerate}[{\rm (i)}]
\item $\Delta^{\overline \psi}$ is a homology $d$-manifold with boundary,

\item $(\beta_0(\Delta^{\overline \psi}), \beta_1(\Delta^{\overline \psi})) =(\beta_0(\Delta), \beta_1(\Delta)+1)$ or $(\beta_0(\Delta)-1, \beta_1(\Delta))$ and

\item $\Delta^{\overline \psi}$ is stacked if and only if $\Delta$ is stacked.
\end{enumerate}
\end{lemma}

\begin{proof}
(i) For every $\alpha \in \Delta^{\overline \psi}$ with $\alpha \not \subseteq \tau$, there is a unique face
$\gamma \in \Delta$ such that $\alpha=\psi_+(\gamma)$ and $\lk_{\Delta^{\overline \psi}}(\alpha)$ is
combinatorially isomorphic to $\lk_{\Delta}(\gamma)$. Thus, to prove the statement, it is enough to show that,
for every $\alpha \subseteq \tau$, $\lk_{\Delta^{\overline \psi}}(\alpha)$ is either a homology $(d-
\dim(\alpha)-1)$-sphere or $(d-\dim(\alpha)-1)$-ball. It is clear that $|\lk_{\Delta^{\overline \psi}}(\tau)|
\cong S^0$. For a proper face $\alpha$ of $\tau$, a straightforward computation implies
\begin{align}
\lk_{\Delta^{\overline \psi}}(\alpha) & = \lk_\Delta(\alpha) \cup_{\psi'} \lk_\Delta(\psi^{-1}(\alpha)), \nonumber
\end{align}
where $\psi': \psi^{-1}(\tau \setminus \alpha ) \to \tau\setminus \alpha$ is the restriction of $\psi$ to
$\psi^{-1}(\tau \setminus \alpha)$. By Lemma \ref{lemma:3.1}, $\lk_{\Delta^{\overline \psi}}(\alpha)$ is a
homology $(d-\dim(\alpha)-1)$-ball.

(ii) It is clear that $\beta_0(\Delta^{\overline \psi})= \beta_0(\Delta)-1$ if $\sigma$ and $\tau$ belong to
different connected components and $\beta_0(\Delta^{\overline \psi})= \beta_0(\Delta)$ if $\sigma$ and $\tau$ are
in the same connected component. Observe that $\widetilde H_i(|\Delta^{\overline \psi}|) \cong \widetilde
H_i(|\Delta^{\overline \psi}|,|\tau|) \cong \widetilde H_i(|\Delta|,|\sigma| \cup |\tau|)$ for all $i$. Then the
desired statement follows from the following exact sequence of pairs
\begin{eqnarray*}
\begin{array}{lllllllllllll}
 0=  \widetilde H_1(|\sigma| \cup |\tau|) &
\longrightarrow & \widetilde H_1(|\Delta|) & \longrightarrow &
\widetilde H_1(|\Delta|,|\sigma| \cup |\tau|) &  &\\
\longrightarrow  \widetilde H_0(|\sigma| \cup |\tau|) &
\longrightarrow & \widetilde H_0(|\Delta|) & \longrightarrow &
\widetilde H_0(|\Delta|,|\sigma| \cup |\tau|) & \longrightarrow & 0.
\end{array}
\end{eqnarray*}

(iii) This statement follows from the proof of (i) since it says that the interior faces of $\Delta^{\overline
\psi}$ are $\tau$ and $\psi_+(\alpha)$ for all interior faces $\alpha$ of $\Delta$.
\end{proof}

The proof of Lemma \ref{lemma:3.2}\,(i) also says that if $\Delta$ is connected then $\partial(\Delta^{\overline
\psi})=(\partial \Delta)^\psi$. Also $|\partial(\Delta_1 \cup_\psi \Delta_2)|$ is a connected sum of $|\partial
\Delta_1|$ and $|\partial \Delta_2|$.

Next, we consider the inverse of the construction of $\Delta^{\overline{\psi}}$, which we call simplicial handle
deletions.

\begin{lemma} \label{lemma:3.3}
Let $B$ be a homology $d$-ball with vertex set $V$, $\sigma$ an interior $(d-1)$-face of $B$ with $\partial
\overline \sigma \subseteq \partial B$. Then $B[{V\setminus \sigma}]$ contains exactly two connected components.
\end{lemma}

\begin{proof}
Let $v \ast \partial B=\partial B \cup \{ \{v\} \cup \alpha : \alpha \in \partial B\}$ be the cone over $\partial
B$, where $v$ is a new vertex. It is easy to see that $S= B \cup (v \ast \partial B)$ is a homology $d$-sphere. Then
\begin{align}
\widetilde H_0(S[{V\setminus \sigma}]) &
\cong
\widetilde H_{d-1}(S[{\sigma \cup \{v\}}])
\cong
\widetilde H_{d-1}(S[{\sigma \cup \{v\}}], (v \ast \partial B)[{\sigma \cup \{v\}}])), \nonumber
\end{align}
where the first isomorphism follows from the Alexander duality \cite[Theorem 6.2.17]{Sp} and the second
isomorphism follows from the long exact sequence of pairs since $\widetilde H_i((v\ast \partial B)[{\sigma \cup
\{v\}}])=0$ for all $i$. Since $B[{V\setminus \sigma}]= S[{V\setminus \sigma}]$ and since
\begin{align}
\widetilde H_{d-1}(S[{\sigma \cup \{v\}}], (v \ast \partial B)[{\sigma \cup \{v\}}])) &
\cong
\widetilde H_{d-1}(B[\sigma], (\partial B)[\sigma]))
=
\widetilde H_{d-1}(\overline \sigma , \partial \overline \sigma) \cong \FF, \nonumber
\end{align}
$B[{V\setminus\sigma}]$ has exactly two connected components.
\end{proof}

Recall that any interior $(d-1)$-face $\sigma$ of a homology $d$-manifold $\Delta$ is contained in exactly two
facets since $\lk_\Delta(\sigma)$ has the same homologies as $S^0$.

\begin{lemma} \label{lemma:3.4}
Let $B$ and $\sigma$ be as in Lemma $\ref{lemma:3.3}$, $C_1$ and $C_2$ the connected components of $B[{V
\setminus \sigma}]$ and let $W_1$ and $W_2$ be the vertex sets of $C_1$ and $C_2$ respectively. Let $B_1=B[{W_1
\cup \sigma}]$ and $B_2=B[{W_2 \cup \sigma}]$. Then the following hold.
\begin{enumerate}[{\rm (i)}]
\item $B=B_1 \cup B_2$ and $B_1 \cap B_2= \overline \sigma$.

\item If $\{x\} \cup \sigma$ and $\{y\}\cup \sigma$ are the facets of $B$ containing $\sigma$, then one of $x$
and $y$ is in $B_1$ and the other is in $B_2$.
\item $B_1$ and $B_2$ are homology $d$-balls.
\end{enumerate}
\end{lemma}

\begin{proof}
(i) It is clear that $B \supseteq B_1 \cup B_2$ and $B_1 \cap B_2 =\overline \sigma$. We prove $B \subseteq B_1
\cup B_2$. Let $\alpha$ be a facet of $B$. Then $\alpha \setminus \sigma \in B[{V\setminus\sigma}]$ is contained
in either $W_1$ or $W_2$, which implies $\alpha \in B_1 \cup B_2$.

(ii) Since $C_1$ and $C_2$ are not empty, there are facets $\alpha,\gamma$ of $B$ such that $\alpha \in B_1$ and
$\gamma \in B_2$.
Since $B$ is a homology $d$-ball, it is a $d$-dimensional
pseudomanifold and hence there is a sequence $\alpha=\alpha_0,\alpha_1,\dots, \alpha_k =\gamma$ of facets such that
$\alpha_{i-1} \cap \alpha_i$ has dimension $d-1$ for $1\leq i\leq k$
(see \cite[p.\ 150 and 278]{Sp}).
Let $j$ be a
number such that $\alpha_{j-1} \in B_1$ and $\alpha_j \in B_2$. Then $\alpha_{j-1} \cap \alpha_{j}$ must be
$\sigma$. Since $\{x \} \cup \sigma$ and $\{y \} \cup \sigma$ are the only facets containing $\sigma$, they must
be $\alpha_{j-1}$ and $\alpha_j$.

(iii) We use induction on $d$. The statement is clear when $d=1$. Consider the subcomplex $B_1$. If $\alpha$ is a
face in $B_1\setminus\overline{\sigma}$ then any facet $\gamma\in B$ containing $\alpha$ must intersect $W_1$ and
hence is in $B_1$. If $\alpha\in \overline{\sigma}$, then $\alpha$ is a face of the $d$-face $\sigma\cup\{x\} \in
B_1$. Thus, $B_1$ is pure. Next, let $v$ be a vertex of $B_1$.  If $v \not \in \sigma$ then $\lk_{B_1}(v)=
\lk_B(v)$ is a homology $(d-1)$-sphere or a homology $(d-1)$-ball. Suppose $v \in \sigma$. Then $\lk_B(v)$ is a
homology $(d-1)$-ball such that $\sigma \setminus \{v\}$ is its interior face. Since $\lk_{\lk_B(v)}(\alpha)=
\lk_B(\{v\}\cup \alpha)$ is a homology ball for any $\alpha \in \partial(\overline{\sigma \setminus \{v\}})$, it
follows that $\partial(\overline{\sigma\setminus\{v\}}) \subseteq \partial(\lk_B(v))$. Since $x$ and $y$ are in
$\lk_B(v)$, $\lk_B(v)[{W_1}]$ and $\lk_B(v)[{W_2}]$ are non-empty. Thus, they are different components of
$\lk_B(v)[{V\setminus\sigma}]$. By the induction hypothesis, $\lk_{B_1}(v)= \lk_\Delta(v)[(W_1 \cup \sigma)\setminus \{v\}]$ is a homology $(d-1)$-ball. Thus, $\lk_{B_1}(v)$ is either a homology $(d-1)$-sphere or a homology $(d-1)$-ball for every vertex $v$ of $B_1$. This implies that $B_1$ is a homology $d$-manifold with boundary. Since part (i) and the exactness of the Mayer--Vietoris sequence imply $\widetilde{H}_i(B_1)=0$ for
all $i$, $B_1$ is a homology $d$-ball by Lemma \ref{lemma:2.1}. Similarly, $B_2$ is a homology $d$-ball.
\end{proof}

We say that $B_i$ in Lemma \ref{lemma:3.4} is the {\em $x$-component} (resp.\ {\em $y$-component}) of $B$ with
respect to $\sigma$ if it contains $x$ (resp.\ $y$).

Let $\Delta$ be a homology $d$-manifold with boundary. Suppose that $\Delta$ has an interior $(d-1)$-face
$\sigma=\{z_1,\dots,z_d\}$ with $\partial \overline \sigma \subseteq \partial \Delta$. Let $\{x \} \cup \sigma$
and $\{y\} \cup \sigma$ be the facets of $\Delta$ containing $\sigma$. Consider
\begin{align}
R & =\{\alpha \in \Delta: \alpha \cap \sigma \ne \emptyset,\ \alpha \not \subseteq \sigma\}. \nonumber
\end{align}
Observe that, for each $\tau \subsetneq \sigma$, $\lk_\Delta(\tau)$ is a homology ball satisfying the assumption
of Lemma \ref{lemma:3.3} in the sense that $\sigma \setminus \tau$ is an interior face of $\lk_{\Delta}(\tau)$
with $\partial(\overline{\sigma \setminus \tau}) \subseteq \partial (\lk_\Delta(\tau))$. Let
\begin{align}
R_x(k) & =\{ \alpha \in R: z_k \in \alpha, \alpha\setminus\{z_k\} \mbox{ is in the $x$-component of }
\lk_\Delta(z_k) \mbox{ w.r.t.\ } \sigma \setminus \{z_k\}\} \nonumber
\end{align}
and define $R_y(k)$ similarly. Let
\begin{align}
X & =\bigcup_{k=1}^d R_x(k) \mbox{ and } Y=\bigcup_{k=1}^d R_y(k). \nonumber
\end{align}
Note that $R=X \cup Y$.

\begin{lemma} \label{lemma:3.5}
If $R_x(k)$, $R_y(k)$, $X$ and $Y$ are as above then $X \cap Y=\emptyset$. Also, $\{\alpha \in X : z_k \in
\alpha\}=R_x(k)$ and $\{\alpha \in Y: z_k \in \alpha\}=R_y(k)$ for $1\leq k\leq d$.
\end{lemma}

\begin{proof}
To prove the first result, we must prove that $R_x(k) \cap R_y(\ell) = \emptyset$ for all $k \ne \ell$.
Suppose to the contrary that $\alpha \in R_s(k) \cap R_y(\ell)$ for some $k \neq \ell$. Then $\alpha\setminus
\{z_k, z_{\ell}\}$ is in the $x$-component and the $y$-component of $\lk_\Delta(\{z_k,z_\ell\})$ with respect to
$\sigma \setminus \{z_k,z_\ell\}$ and hence $\alpha \subseteq \sigma$, a contradiction since $\alpha \in R$.

Let $\alpha \in X$ with $z_k \in \alpha$. Then $\alpha \in R_x(\ell)$ for some $\ell$. If $\ell=k$ then $\alpha
\in R_x(k)$. Otherwise, $\alpha\setminus \{z_k, z_{\ell}\}$ and $x$ are in the same component of
$\lk_{\Delta}(\{z_k, z_\ell\})$. Since $\lk_{\Delta}(z_k) \supseteq \lk_{\Delta}(\{z_k,z_\ell\})$, we have $\alpha
\in R_x(k)$. This proves that $\{\alpha \in X : z_k \in \alpha\}=R_x(k)$. Similarly, $\{\alpha \in Y : z_k \in
\alpha\}=R_y(k)$.
\end{proof}

\begin{definition} \label{def:handel.deletion}
Let $\Delta$ be a homology $d$-manifold with boundary and let $\sigma=\{z_1,\dots,z_d\}$ be an interior $(d-
1)$-face of $\Delta$ with $\partial \overline \sigma \subseteq \partial \Delta$. Let $R$, $R_x(k)$, $R_y(k)$, $X$
and $Y$ be as above. Let $z_1^+,\dots,z_d^+$ be new vertices and $\sigma^+=\{z_1^+,\dots,z_d^+\}$. For $\alpha
=\alpha^{\,\prime} \cup\{z_{i_1},\dots,z_{i_\ell}\} \in X$ with $\alpha^{\,\prime}\cap \sigma=\emptyset$, define
$\alpha^+=\alpha^{\,\prime} \cup\{z_{i_1}^+,\dots,z_{i_\ell}^+\}$. Consider the simplicial complex
\begin{align}
\widetilde \Delta^\sigma & = \{\alpha \in \Delta: \alpha \not \in X\} \cup \{\alpha^+ : \alpha \in X\} \cup
\overline {\sigma^+}. \nonumber
\end{align}
We say that $\widetilde \Delta^\sigma$ is obtained from $\Delta$ by a {\em simplicial handle deletion} over
$\sigma$.
\end{definition}

Intuitively, $\widetilde \Delta^\sigma$ is a simplicial complex obtained from $\Delta$ by cutting it
along the face $\sigma$. Note that this construction is a simplified version of the construction in
\cite[Lemma 3.3]{BDlow}. Also, a similar construction for manifolds without boundary was considered by Walkup \cite{Wa}. Simplicial handle deletion has the following property.

\begin{theorem} \label{theo:3.7}
Let $\widetilde \Delta^\sigma$ be obtained from a homology $d$-manifold with boundary $\Delta$ by a {\em
simplicial handle deletion} over $\sigma$. Then
\begin{enumerate}[{\rm (i)}]
\item $\widetilde \Delta^\sigma$ is a homology $d$-manifold with boundary, and \item $\Delta=( \widetilde
\Delta^\sigma)^{\overline \psi}$, where $\psi:\sigma^+ \to \sigma$ is the bijection given by $\psi(z_i^+)= z_i$
for all $i$.
\end{enumerate}
\end{theorem}

\begin{proof}
The second statement is straightforward if $\widetilde \Delta^\sigma$ is a homology manifold. So, we prove (i).
For simplicity, we write $\widetilde \Delta=\widetilde \Delta^\sigma$. Let $V$ be the vertex set of $\Delta$.

We prove that each vertex link of $\widetilde \Delta$ is either a homology $(d-1)$-sphere or a $(d-1)$-ball.
Suppose $v \not \in \sigma^+ \cup \sigma$. Define the map $\varphi: \Delta \to \widetilde \Delta$ by
$\varphi(\alpha)=\alpha$ if $\alpha \not \in X$ and $\varphi(\alpha)=\alpha^+$ if $\alpha \in X$. Then $\varphi$
gives a bijection between $\Delta \setminus\overline \sigma$ and $\widetilde \Delta \setminus (\overline
{\sigma^+} \cup \overline \sigma)$, in particular, it gives a bijection between $\{\alpha: v \in \alpha \in
\Delta\}$ and $\{\alpha : v \in \alpha \in \widetilde \Delta\}$. Thus $\lk_{\widetilde \Delta}(v)$ is
combinatorially isomorphic to $\lk_\Delta(v)$, which implies the desired property. Suppose $v =z_k^+$ for some
$k$. Then
\begin{align}
\lk_{\widetilde \Delta}(v) & =\lk_{\widetilde \Delta}(z_k^+)= \overline{(\sigma \setminus \{z_k\})^+} \cup \{
(\alpha \setminus \{z_k\})^+ : z_k \in \alpha \in X\}. \nonumber
\end{align}
On the other hand, the $x$-component of $\lk_{\Delta}(z_k)$ is
\begin{align}
\overline {\sigma \setminus \{z_k\}} & \cup \{ (\alpha \setminus \{z_k\}): z_k \in \alpha \in R_x(k)\}. \nonumber
\end{align}
By Lemma \ref{lemma:3.5}, they are combinatorially isomorphic. This proves that $\lk_{\widetilde \Delta}(v)$ is a
homology $(d-1)$-ball. Finally, suppose $v =z_k$ for some $k$. Since $X \cap Y = \emptyset$,
\begin{align*}
\lk_{\widetilde \Delta}(v)
&= \overline{\sigma \setminus \{v\}} \cup \{ \alpha \setminus \{v\}: v \in \alpha \in R_y(k) \setminus X\}\\
& = \overline{\sigma \setminus \{z_k\}} \cup \{ \alpha \setminus \{z_k\}: \alpha
\in R_y(k)\} \nonumber
\end{align*}
is the $y$-component of $\lk_\Delta(v)$ w.r.t.\ $\sigma\setminus\{v\}$. Thus $\lk_{\widetilde \Delta}(v)$
is a homology $(d-1)$-ball.

Finally, $\widetilde \Delta$ has a non-empty boundary since $\sigma \in \partial \widetilde \Delta$.
\end{proof}

\begin{remark}
In this section, we consider simplicial handle deletions for homology manifolds.
One may ask if the result holds also for combinatorial manifolds. However, we are not sure if a simplicial handle deletion preserves being combinatorial manifolds.
This is because, in Lemma \ref{lemma:3.4}, we are not sure if $B_1$ and $B_2$ are PL-balls when $B$ is a PL-ball.
\end{remark}

\section{A characterization of stacked manifolds}

In this section, we present a characterization of stacked manifolds.
We first define an analogue of Walkup's class for manifolds with boundary.

\begin{definition} \label{def:Walkup.boundary}
Let $d \geq 2$ be an integer. We recursively define $ \overline{\mathcal H^d}(k)$ as follows.
\begin{enumerate}[{\rm (a)}]
\item $\overline{\mathcal H^d}(0)$ is the set of stacked triangulations of $d$-balls.

\item $\Delta$ is a member of $\overline{\mathcal H^d}(k+1)$ if it is obtained from a member of
$\overline{\mathcal H^d}(k)$ by a simplicial handle addition.
\end{enumerate}
Let $\overline{\mathcal H^d}=\bigcup_{k \geq 0} \overline{\mathcal H^d}(k)$.
\end{definition}

Note that every stacked triangulation of the $d$-ball is obtained from a $d$-simplex by taking a simplicial
connected union with a $d$-simplex repeatedly. See \cite[Lemma 2.1]{DS}. The classes  $\overline{\mathcal H^d}$ and
$\mathcal H^d$ have the following simple relation.

\begin{lemma}\label{lemma:3.9}
For all integers $d \geq 3$ and $k \geq 0$, one has $\mathcal H^d(k)= \{ \partial \Delta: \Delta \in
\overline{\mathcal H^d}(k)\}$.
\end{lemma}

\begin{proof}
The case when $k=0$ and the inclusion $\mathcal H^d(k) \supseteq \{ \partial \Delta: \Delta \in
\overline{\mathcal H^d}(k)\}$, for all $k\geq 0$, are obvious. For $k>0$, the converse inclusion follows by
induction on $k$. Indeed, if $\Gamma \in \mathcal H^{d}(k)$, then by induction we may assume that there is a
$\Delta \in \overline{\mathcal H^d}(k-1)$ such that $\Gamma= (\partial \Delta)^\psi$ for some admissible
bijection $\psi: \sigma \to \tau$ in $\partial \Delta$. Since, $\partial \Delta$ and $\Delta$ have the same
$1$-faces by Lemma \ref{lemma:3.2}(iii), the bijection $\psi$ is also admissible for $\Delta$, and $\Gamma=
\partial (\Delta^{\overline \psi}) \in \{ \partial \Delta: \Delta \in \overline{\mathcal H^d}(k)\}$.
\end{proof}

\begin{lemma} \label{lemma:3.10}
If $\Delta \in \overline{\mathcal H^d}(k)$ and $\Gamma \in \overline{\mathcal H^d}(\ell)$ then their simplicial
connected union belongs to $\overline{\mathcal H^d}(k+\ell)$.
\end{lemma}

\begin{proof}
We may assume $k\leq \ell$. We use induction on $k+\ell$. If $k +\ell=0$ then the assertion follows from Lemma
\ref{lemma:3.2}(iii). Suppose $k+\ell >0$. Then $\Gamma=\Sigma^{\overline \varphi}$ for some $\Sigma \in
\overline{\mathcal H^d}(\ell-1)$ and for some admissible bijection $\varphi$ between facets of $\partial \Sigma$.
Let $\psi$ be a bijection from a facet of $\partial \Delta$ to a facet of $\partial \Gamma$. Then $\Delta
\cup_\psi \Gamma$ is $(\Delta \cup_\psi \Sigma)^\varphi$ (by an appropriate identification of the vertices). By
induction hypothesis, we have $\Delta \cup_\psi \Sigma \in \overline{\mathcal H^d}(k+\ell-1)$ and hence $\Delta
\cup_\psi \Gamma \in \overline{\mathcal H^d}(k+\ell)$.
\end{proof}

\begin{remark}
A similar result for $\mathcal H^d$ was proved by Walkup \cite[Proposition 4.4]{Wa}.
\end{remark}

\begin{theorem} \label{theo:main.3.12}
For $d\geq 2$, let $\Delta$ be a connected homology $d$-manifold with boundary. Then $\Delta$ is
stacked if and only if $\Delta \in \overline{\mathcal H^d}$.
\end{theorem}

\begin{proof}
The `only if part'  is obvious if $\Delta$ has one facet. Suppose that $\Delta$ has more than one facet. Then
$\Delta$ has an interior $(d-1)$-face $\sigma$. Since $\Delta$ is stacked, it has no interior faces of dimension
$\leq d-2$. Thus we have $\partial \overline \sigma \subseteq \partial \Delta$. By Lemma \ref{lemma:3.2} and
Theorem \ref{theo:3.7}, $\Delta$ is a simplicial connected union of two connected stacked manifolds or
is obtained from a connected stacked manifold having a smaller first Betti number by a simplicial handle
addition. Then the assertion follows by induction on the number of interior $(d-1)$-faces.
\end{proof}

By Lemma \ref{lemma:3.9} and Theorem \ref{theo:main.3.12} we obtain the following.

\begin{corollary} \label{cor:main.3.13}
Let $\Delta$ be a connected, closed, homology manifold of dimension $d \geq 2$. Then $\Delta$ is stacked if and
only if $\Delta \in \mathcal H^{d+1}$.
\end{corollary}

Finally, we discuss a connection between Corollary \ref{cor:main.3.13} and a question posed by Novik and Swartz \cite{NS}. Novik and Swartz \cite[Theorem 5.2]{NS} gave the
following interesting characterization of members of Walkup's class $\mathcal{H}^{d+1}$ for $d\geq 4$.

\begin{proposition}[Novik--Swartz] \label{prop:NS2009}
Let $\Delta$ be a connected, closed, orientable, homology manifold of dimension $d \geq 3$. Then
\begin{align}
f_1(\Delta) -(d+1) f_0(\Delta) + \binom{d+2}{2} \geq \binom{d+2}{2} \beta_1(\Delta). \nonumber
\end{align}
Further, if $d\geq 4$ then $f_1(\Delta) -(d+1) f_0(\Delta) + \binom{d+2}{2} = \binom{d+2}{2} \beta_1(\Delta)$
if and only if $\Delta\in \mathcal{H}^{d+1}$.
\end{proposition}

It was asked by Novik and Swartz \cite[Problem 5.3]{NS} if the equality case of the last statement in Proposition
\ref{prop:NS2009} also holds in dimension $3$.
Corollary \ref{cor:main.3.13}  and the next result of Bagchi \cite[Theorem 1.14]{Bag}
answer this question.

\begin{proposition}[Bagchi] \label{prop:Ba2016}
Let $\Delta$ be a connected, closed, homology $3$-manifold.
Then $f_1(\Delta) - 4f_0(\Delta) + 10 = 10 \beta_1(\Delta)$ if and only if $\Delta$ is stacked.
\end{proposition}

\begin{corollary} \label{cor:NS2009.problem}
Let $\Delta$ be a connected, closed, homology $3$-manifold. The following conditions are equivalent.
\begin{enumerate}[{\rm (i)}]
\item $\Delta$ is a member of $\mathcal{H}^{4}$.

\item $\Delta$ is stacked.

\item $f_1(\Delta) -4 f_0(\Delta) + 10 = 10 \beta_1(\Delta)$.
\end{enumerate}
\end{corollary}

\begin{remark}
It is known that the topological type of a member of $\mathcal H^{d+1}$ is one of the following: (i) the
$d$-sphere $S^d$ (ii) connected sums of sphere product $S^{d-1} \times S^1$ (iii) connected sums of twisted
sphere product $S^{d-1} \mbox{$\times \hspace{-8.2pt} _-$} S^1$. See \cite[Section 3]{LSS}. Thus Corollary
\ref{cor:main.3.13} also gives a new restriction on the topological types of stacked manifolds in
dimensions 2 and 3.
\end{remark}

\begin{remark}
\label{nonorientable}
Recently, Proposition \ref{prop:NS2009} was extended to non-orientable homology manifolds (and even to normal pseudomanifolds) by the second author. See \cite[Theorem 5.3]{Mu}.
\end{remark}

\section{Tight triangulations and stackedness}

In this section, we study stackedness of tight triangulations.
For a simplicial complex $\Delta$ with vertex set $V$, a subset $\sigma \subseteq V$ of $k+1$
elements is called a \emph{missing $k$-face} of $\Delta$ if $\sigma \notin \Delta$ and all proper subsets of
$\sigma$ are faces of $\Delta$. If $\sigma$ is a missing $k$-face of $\Delta$, then we have $\widetilde
H_{k-1}(\Delta[\sigma]) \cong \FF$. The following lemma follows from the definition of tightness.

\begin{lemma} \label{lemma:4.1}
Let $\Delta$ be a tight simplicial complex on the vertex set $V$. Then
\begin{enumerate}[{\rm (i)}]
\item for all subsets $U \subseteq W$ of $V$, the natural map $\widetilde H_i(\Delta[U]) \to \widetilde
H_i(\Delta[W])$ induced by the inclusion is injective, and

\item if $\beta_{k-1}(\Delta)=0$ then $\Delta$ has no missing $k$-faces.
\end{enumerate}
\end{lemma}

For a simplicial complex $\Delta$, we identify its {\em $1$-skeleton} ${\rm Skel}_1(\Delta) = \{\sigma \in \Delta
: \dim(\sigma) \leq 1\}$ with the simple graph whose vertex set is the set of the vertices of $\Delta$ and whose
edge set is the set of the edges (1-simplices) in $\Delta$. We say that a simple graph $G$ is {\em chordal} if it
has no induced cycle of length $\geq 4$. The following result is due to Kalai \cite[Theorem 8.5]{Ka}.

\begin{proposition}[Kalai] \label{prop:Ka1987}
Let $\Delta$ be a homology $(d-1)$-sphere with $d \geq 3$. Then $\Delta$ is stacked if and only if the
$1$-skeleton of $\Delta$ is chordal and $\Delta$ has no missing $k$-faces for $1 < k < d-1$.
\end{proposition}

Let $\Delta$ be a closed homology manifold of dimension $d \geq 3$. We say that $\Delta$ is {\em
tight-neighborly} if $\binom{f_0(\Delta) -d -1}{2}= \binom{d+2}{2} \beta_1(\Delta;\FF)$. Since $\binom{f_0}{2}
-(d+1)f_0 + \binom{d+2}{2}= \binom{f_0 -d -1}{2}$, $\Delta$ is tight-neighborly if and only if $\Delta$ is
stacked and neighborly by Proposition \ref{prop:NS2009} (see Remark \ref{nonorientable} for the non-orientable case). Note that the latter condition says that tight-neighborliness does not depend on the choice of a field $\FF$. Here we prove the following.

\begin{theorem} \label{theo:main.4.3}
Let $\Delta$ be a tight, connected, closed, homology manifold of dimension $d \geq 4$ such that $\beta_i(\Delta)=0$
for $1 <i < d-1$. Then $\Delta$ is locally stacked.
\end{theorem}

\begin{proof}
Let $v$ be a vertex of $\Delta$. We prove that $\lk_\Delta(v)$ is stacked.

We first claim that no induced subcomplex of $\lk_\Delta(v)$ can be a $1$-dimensional simplicial complex which
forms a cycle. Suppose to the contrary that $\lk_\Delta(v)[W]$ is a cycle for some $W$. Let $C=\lk_\Delta(v)[W]$
and $v\ast C=C \cup \{ \{v\} \cup \sigma: \sigma \in C\}$. Then we have $\Delta[{W \cup \{v\}}]= \Delta[W] \cup
(v\ast C)$ and $\Delta[W] \cap (v\ast C)=C$. Consider the Mayer--Vietoris exact sequence
\begin{align}
\widetilde H_2(\Delta[{W \cup \{v\}}]) & \longrightarrow \widetilde H_1(C) \longrightarrow \widetilde
H_1(\Delta[W]) \oplus \widetilde H_1(v\ast C) \stackrel \varphi \longrightarrow \widetilde H_1(\Delta[{W \cup
\{v\}}]). \nonumber
\end{align}
Since $\Delta$ is tight and $\beta_2(\Delta)=0$, we have $\widetilde H_2(\Delta[{W\cup\{v\}}])=0$. Then since
$\widetilde H_1(C) \ne 0$, the map $\varphi$ has a non-trivial kernel. However, since $\widetilde H_1(v\ast C)=0$,
this contradicts the tightness of $\Delta$ as it implies that $\varphi$ is injective by Lemma \ref{lemma:4.1}(i).
Hence no induced subcomplex of $\lk_\Delta(v)$ can be a cycle.

Now we prove the statement. By Lemma \ref{lemma:4.1}(ii), $\Delta$ has no missing $k$-faces for $2<k<d$. This
implies that $\lk_\Delta(v)$ has no missing $k$-faces for $2 < k <d-1$. Also, $\lk_\Delta(v)$ has no missing
$2$-faces since if it has a missing $2$-face $\sigma$ then $\lk_\Delta(v)[\sigma]$ is a cycle of length $3$.
Similarly, the $1$-skeleton of $\lk_\Delta(v)$ is a chordal graph since if it has an induced cycle of length
$\geq 4$ with the vertex set $W$, then $\lk_\Delta(v)[W]$ is a cycle. Thus, by Proposition \ref{prop:Ka1987},
$\lk_\Delta(v)$ is stacked.
\end{proof}

For any field, a tight homology manifold is orientable (cf.\ \cite{BDtight}). From Theorem \ref{theo:main.4.3}
and all the known results, we have the following.

\begin{corollary}
\label{cor4.4}
Let $\Delta$ be a closed, orientable, homology manifold of
dimension $d\geq 4$. Then the following are equivalent.
\begin{enumerate}[{\rm (i)}]
\item $\Delta$ is tight-neighborly.

\item $\Delta$ is a neighborly member of $\mathcal{H}^{d+1}$.

\item $\Delta$ is neighborly and stacked.

\item $\Delta$ is neighborly and locally stacked.

\item $\Delta$ is tight and $\beta_i(\Delta)=0$ for $1< i< d-1$.
\end{enumerate}
\end{corollary}

\begin{proof}
The equivalence (i) $\Leftrightarrow$ (ii) follows from Proposition  \ref{prop:NS2009}, (ii) $\Leftrightarrow$
(iii) follows from Corollary \ref{cor:main.3.13}, and (ii) $\Leftrightarrow$ (iv) follows from Kalai's result
\cite[Corollary 8.4]{Ka}. Now, (v) $\Rightarrow$ (iv) follows from Theorem \ref{theo:main.4.3}. Since $\Delta \in
\mathcal{H}^{d+1}$ implies $\beta_i(\Delta)=0$ for $1 < i < d-1$, (ii) \& (iv) $\Rightarrow$ (v) follows from
\cite[Theorem 3.11]{BDtight}. This completes the proof.
\end{proof}

From the equivalence of (i) and (v) in Corollary \ref{cor4.4} it follows that tight triangulations of connected
sums of $S^{d-1}$-bundles over $S^1$ are tight-neighborly for $d \geq 4$. This answers a question asked by
Effenberger \cite[Question 4.5]{Ef}.

It would be natural to ask if the results in this section hold in dimension $3$.
Very recently, Bagchi, Spreer and the first author \cite{BDS2}
proved the following result (which answers a question asked in a previous version of this paper).

\begin{proposition}[Bagchi--Datta--Spreer] \label{prop:BDS2016}
A closed triangulated $3$-manifold $M$ is $\FF$-tight if and only if $M$ is $\FF$-orientable, neighbourly and
stacked.
\end{proposition}

As a consequence of Proposition \ref{prop:BDS2016} and Corollary
\ref{cor:NS2009.problem} we get the following (compare Corollaries 1.4 and 1.5 of \cite{BDS2}).

\begin{corollary} \label{cor:BDS2016}
Let $\Delta$ be an $\FF$-orientable, closed, triangulated $3$-manifold. The following conditions are equivalent.
\begin{enumerate}[{\rm (i)}]
\item $\Delta$ is $\FF$-tight.

\item $\Delta$ is neighborly and stacked.

\item $\Delta$ is a neighborly member of $\mathcal{H}^{4}$.

\item $(f_0(\Delta)-4)(f_0(\Delta)-5) = 20\beta_1(\Delta; \FF)$.
\end{enumerate}
\end{corollary}

\medskip

\noindent \textbf{Acknowledgements}: The authors
thank Bhaskar Bagchi for many useful comments on a preliminary version of this paper. They also thank Eran Nevo
for letting them know that Lemma \ref{lemma:3.3} follows from the Alexander duality. The first author was
partially supported by DIICCSRTE, Australia (project AISRF06660) \& DST, India (DST/INT/AUS/P-56/2013 (G)) under
the Australia-India Strategic Research Fund and by the UGC Centre for Advanced Studies (F.510/6/CAS/2011
(SAP-I)). The second author was partially supported by JSPS KAKENHI 26400043.

\end{document}